\documentclass[a4paper,12pt,reqno]{amsart}
\usepackage{amsfonts,amssymb,hyperref,amsthm,enumerate,
color,stmaryrd,pgf,tikz,comment,multirow}
\usepackage[left=2.6cm,right=2.6cm,top=3cm,bottom=3cm,bindingoffset=0cm]
{geometry}
\newtheorem{theorem}{Theorem}[section]
\newtheorem{lemma}[theorem]{Lemma}
\newtheorem{proposition}[theorem]{Proposition}
\newtheorem{assumption}[theorem]{Assumption}
\theoremstyle{remark}

\theoremstyle{definition}

\newtheorem{corollary}[theorem]{Corollary}
\newtheorem{problem}{Problem}

\def\F{\mathbb{F}}
\def\Z{\mathbb{Z}}
\DeclareMathOperator{\aut}{Aut}
\DeclareMathOperator{\Ker}{Ker}
\DeclareMathOperator{\sym}{Sym}
\DeclareMathOperator{\orb}{Orb}
\DeclareMathOperator{\core}{Core}
\DeclareMathOperator{\agl}{AGL}
\DeclareMathOperator{\ch}{~char~}
\DeclareMathOperator{\id}{id}
\DeclareMathOperator{\lcm}{lcm}
\DeclareMathOperator{\caym}{CM}
\newcommand{\sg}[1]{\langle {#1}\rangle}

\newcommand{\blue}[1]{{\color[rgb]{0,0,1}{#1}}}
\title[Cyclic groups underling only smooth skew morphisms]
{Classification of cyclic groups underlying only smooth skew morphisms}
\author[K.~Hu, I.~Kov\'acs \and Y.~S.~Kwon]
{Kan Hu$^{*}$, Istv\'an Kov\'acs \and Young Soo Kwon}
\address{K. Hu,
\newline\indent
Department of Mathematics, Zhejiang Ocean University, Zhoushan, Zhejiang 316022, P.R. China
\newline\indent UP FAMNIT, University of Primorska, Glagolja\v ska 8,
6000 Koper, Slovenia
\newline\indent UP IAM, University of Primorska, Muzejski trg 2,
6000 Koper, Slovenia}
\email{kan.hu@famnit.upr.si}

\address{I. Kov\'acs,
\newline\indent UP IAM, University of Primorska, Muzejski trg 2,
6000 Koper, Slovenia
\newline\indent UP FAMNIT, University of Primorska, Glagolja\v ska 8,
6000 Koper, Slovenia}
\email{istvan.kovacs@upr.si}

\address{Y. S. Kwon,
\newline\indent
Department of Mathematics, Yeungnam University, Kyongsan 712-749, Republic of Korea}
\email{ysookwon@ynu.ac.kr}
\thanks{The first and the second authors were supported by the Slovenian Research Agency (ARRS), the first author by the research project N1-0208 and the second author by 
the research program P1-0285 and research projects J1-1695, N1-0140, J1-2451, N1-0208 and J1-3001. The third author was supported by the Basic Science Research
Program through the National Research Foundation of Korea (NRF)
funded by the Ministry of Education (2018R1D1A1B05048450).
\newline\indent
$^{\ast}$ Corresponding author e-mail: kan.hu@famnit.upr.si
}
\keywords{skew morphism, group factorization, solvable group}
\subjclass[2010]{05C10, 05C25, 57M15}
\begin{document}
\maketitle

\begin{abstract}
A skew morphism of a finite group $A$ is a permutation $\varphi$ of $A$ fixing the
identity element and for which there is an integer-valued function $\pi$ on
$A$ such that $\varphi(ab)=\varphi(a)\varphi^{\pi(a)}(b)$ for all $a, b \in A$.
A skew morphism $\varphi$ of $A$ is smooth if the associated power function $\pi$
is constant on the orbits of $\varphi$, that is, $\pi(\varphi(a))\equiv\pi(a)\pmod{|\varphi|}$ for 
all $a\in A$. In this paper we show that every skew morphism of a cyclic group of order $n$ 
is smooth if and only if $n=2^en_1$, where $0 \le e \le 4$ and $n_1$ is an odd square-free number. 
A partial solution to a similar problem on non-cyclic abelian groups is also given.
\end{abstract}
\section{Introduction}\label{sec:intro}
A \emph{skew morphism} of a group $A$ is a permutation $\varphi$ 
of $A$ fixing the identity 
element of $A$ and for which there exists a function $\pi: A \to \Z$ such that
\[
\varphi(ab)=\varphi(a)\varphi^{\pi(a)}(b)~\text{for all}~a,b\in A.
\]
The function $\pi$ is referred to as the \emph{power function} associated with
$\varphi$. If $\varphi$ is fixed, then the values of $\pi$ are uniquely determined
modulo $|\varphi|$, where $|\varphi|$ denotes the order of $\varphi$, and therefore,
$\pi$ may also be defined as a function from $A$ to $\Z_{|\varphi|}$.
It is a trivial observation that if $\pi(a)=1$ for all $a \in A$, then $\varphi$ is an 
automorphism of $A$. Skew morphisms which are not automorphisms are called
\emph{proper}.

The concept of a skew morphism was introduced by Jajcay and
\v{S}ir\'a\v{n}~\cite{JS2002} to
characterize regular Cayley maps. Without going into the details, for
a Cayley map $M=\caym(A,X,P)$ of a group $A$, where $X$ is
an inverse-closed generating subset of $A$ and $P$ is a cyclic permutation of $X$,
it was shown that $M$ is regular if and only if $P$ extends to a skew morphism
of $A$~\cite[Theorem 1]{JS2002}.

Skew morphisms are closely related  to complementary factorizations of finite groups with a cyclic
factor~\cite{CJT2016}. By a \emph{complementary factorization} of a group $G$ we mean
that $G$ can be written as a product $AB$, where $A$ and $B$ are subgroups of $G$
and $A \cap B=1$. Suppose that $\varphi$ is a skew morphism of $A$. For $a \in A$, let $L_a$ 
denote the permutation of $A$ acting as $L_a(x)=ax$ for all $x \in A$, and set $L_A=\{L_a : a \in A\}$.
It is not difficult to show that the permutation group $\sg{L_A,\varphi}$ of $A$ admits the 
complementary factorization $L_A \sg{\varphi}$. This group is called the \emph{skew product group} 
induced by $A$ and $\varphi$~\cite{CJT2016}.

Conversely, suppose that $G$ is any group admitting a complementary factorization 
$G=AY$, where $Y$ is a cyclic subgroup and it is core-free in $G$.
Fix a generator $y$ of $Y$. 
\blue{Then} for every $a \in A$,  there is a unique element $b \in A$ 
and a unique number $j\in\mathbb{Z}_{|y|}$ such that $ya=b y^j$.  Define the mappings
$\varphi : A \to A$ and $\pi : A \to \mathbb{Z}_{|y|}$ by
\begin{equation}\label{eq:c-rule}
\varphi(a)=b~\text{and}~\pi(a)=j~\stackrel{\text{def}}{\iff}~ya=b y^j~
\text{for all}~a \in A.
\end{equation}
Then $\varphi$ and $\pi$ are well-defined, $\varphi$ is a skew-morphism of $A$,
and $\pi$ is the power function associated with $\varphi$~\cite[Proposition~3.1(a)]{CJT2016}.
\medskip

Regular Cayley maps, skew morphisms and skew product groups for a given 
infinite family of groups have been intensively investigated.
Generally speaking, this seems a challenging problem, because even for the cyclic groups a
complete classification of the skew morphisms is not at hand, apart from the celebrated classification
of regular Cayley maps~\cite{CT2014}, and a partial classification of the skew morphisms 
and the skew product groups~\cite{BACH2022,BJ2017,DH2019,HKZ2021,KN2011,KN2017}.
For the elementary abelian $p$-groups, a characterization of the skew product groups
and a complete classification of the regular Cayley maps can be found  in~\cite{DLY2022,DYL2022}.
For the dihedral groups, the regular Cayley maps have been classified~\cite{KK2021} and a 
characterization of the skew product groups was given
recently by the authors~\cite{HKK2022}.
For the non-abelian simple groups (or more explicitly, the monolithic groups),
and non-abelian characteristically simple groups,
the skew morphisms and skew product groups have been classified,
in contrast to the fact that not much is known
about the regular Cayley maps~\cite{BCV2022, CDL2022}.
\medskip

In this paper we shall continue the investigation of skew morphisms of cyclic groups.
A skew morphism $\varphi$ of a group $A$ is called \emph{smooth}
if the associated power function $\pi$ satisfies the following condition:
\begin{equation}\label{eq:smooth}
\pi(\varphi(a))\equiv\pi(a) \!\!\! \! \pmod {|\varphi|}~\text{for all}~a \in A.
\end{equation}
It is clear that every automorphism is smooth, however, the converse is not true in
general. Smooth skew morphisms were defined by Hu~\cite{H2012} and independently by
Bachrat\'y and Jajcay \cite{BJ2017} under the name of \emph{coset-preserving} skew morphisms.
The smooth skew morphisms of cyclic groups and dihedral groups have been classified by
Bachrat\'y and Jajcay~\cite{BJ2017} (see also \cite[Theorem 16]{HNWY2019} for
an alternative proof) and Wang et al.~\cite{WHYZ2019}, respectively.
To elaborate a subtler relationship between automorphisms and smooth skew
morphisms we need to introduce the concept of a \emph{reciprocal pair} of 
skew morphisms of cyclic groups. Suppose that $(\varphi,\tilde\varphi)$ is a pair of
skew morphisms $\varphi$ and $\tilde\varphi$ of the cyclic groups $\Z_m$ and
$\Z_n$, and $\pi$ and $\tilde\pi$ are the associated power functions 
of $\varphi$ and $\tilde\varphi$, respectively. Then $(\varphi,\tilde\varphi)$ is called
 \emph{reciprocal} if they satisfy the following conditions:
\begin{enumerate}[\rm(a)]
\item $|\varphi|$ divides $n$ and $|\tilde\varphi|$ divides $m$;
\item for all $x \in \Z_m$ and $y \in \Z_n$,
\begin{eqnarray*}
\pi(x) &\equiv& -\tilde\varphi^{-x}(-1) \!\!\!\! \pmod {|\varphi|}, \\
\tilde\pi(y) &\equiv& -\varphi^{-y}(-1) \!\!\!\! \pmod  {|\tilde\varphi|}.
\end{eqnarray*}
\end{enumerate}
It is shown in~\cite[Theorem 3.5]{FHNSW2019}  that, for fixed $m$ and $n$, the isomorphism classes
of regular dessins with complete bipartite underlying graphs $K_{m,n}$ are in one-to-one 
correspondence with the reciprocal pairs $(\varphi,\tilde\varphi)$ of skew morphisms of $\Z_m$ and
$\Z_n$, respectively; moreover, if one of the skew morphisms is an automorphism, then the other is 
necessarily smooth~\cite[Lemma 12]{HNWY2019}.

It was shown by Kov\'acs and Nedela~\cite[Theorem 6.3]{KN2011} that every skew morphism of the 
cyclic groups $\mathbb{Z}_n$ is an automorphism if and only if $n=4$ or $\gcd(n,\phi(n))=1$, 
where $\phi$ is Euler's totient function. Recently Bachrat\'y~\cite{BACH2022} showed that every skew morphism 
of $\Z_n$ is smooth whenever $n=pq$, $4p$, $8p$, $16p$ or $pqr$, 
where $p, q, r$ are distinct primes.  Our first theorem is the following generalization.

\begin{theorem}\label{main1}
Every skew morphism of the cyclic group $\Z_n$ is smooth if and only if $n=2^en_1$, where
$0 \le e \le 4$ and $n_1$ is an odd square-free number.
\end{theorem}

The proof of Theorem~\ref{main1} relies on some preliminary results from group theory as well as 
the theory of skew morphisms, which will be collected in next section.

\medskip

Conder et al.~\cite{CJT2016} classified the non-cyclic abelian groups with the property
that all of their skew morphisms are automorphisms. It was shown that these are precisely 
the elementary abelian $2$-groups~\cite[Theorem~7.5]{CJT2016}. 
In this paper we also propose the investigation of non-cyclic abelian groups 
with the property that all of their skew morphisms are smooth.

\begin{problem}\label{PROB}
Classify the finite non-cyclic abelian groups which underly only smooth skew morphisms.
\end{problem}

In Section~\ref{sec:main2} we give the following partial answer.
\begin{theorem}\label{main2}
Let $A$ be a non-cyclic  abelian group of order $n=2^fn_1$, where $f\geq 0$ and $n_1$ is odd.
If $A$ underlies only smooth skew morphisms, then $n_1$ is square-free, and the Sylow $2$-subgroup
of $A$ contains no direct factors isomorphic to $\Z_{2^e}$ $(e\geq 5)$.
\end{theorem}

\section{Preliminaries}\label{sec:known}

\subsection{Group theory}
All groups in this paper will be finite.  We denote the identity element of a group $G$
by $1_G$ and its order by $|G|$. The order of an element $g \in G$ is denoted by $|g|$.
Let $H$ be a subgroup of a group $G$. The \emph{core of $H$ in $G$} is
the largest normal subgroup of $G$ contained in $H$; in the case when this is trivial,
$H$ is called \emph{core-free} in $G$. Moreover, $C_G(H)$ denotes the \textit{centralizer} of $H$ in $G$. 
If $p$ is prime, then the largest normal $p$-subgroup of $G$ is denoted by $O_{p}(G)$. 
The largest normal nilpotent subgroup of $G$ is the \emph{Fitting subgroup} of $G$, denoted by
$F(G)$.
\begin{proposition}[{\rm\cite[7.4.3, 7.4.7]{Scott1964}}]\label{FITTING}
Let $G$ be a finite group, and $\wp=\{p_1,\ldots,p_k\}$ the set of all prime divisors of $|G|$. Then
\begin{enumerate}[\rm(a)]
\item $F(G)=\prod_{i=1}^{k}O_{p_i}(G)$. 
\item If $G$ is a solvable group, then $C_G(F(G))\leq F(G)$.
\end{enumerate}
\end{proposition}

A group $G$ is called \emph{supersolvable} if there is a sequence
\[
1=N_0 < N_1 < \cdots < N_k=G
\]
 of normal subgroups of $G$ such that $|N_i/N_{i-1}|$ is a prime for every $i$, $1 \le i \le k$.
The following results are well known..

\begin{proposition}[{\rm \cite[13.2.9, 13.3.1]{Scott1964}}]\label{AB}
Suppose that a group $G$ has a factorization $G=AB$.
\begin{enumerate}[{\rm (a)}]
\item If both $A$ and $B$ are nilpotent, then $G$ is solvable.
\item If both $A$ and $B$ are cyclic, then $G$ is supersolvable.
\end{enumerate}
\end{proposition}

The following fact is often referred to as the \textit{Sylow tower property} of supersolvable groups.

\begin{proposition}[{\rm \cite[7.2.19]{Scott1964}}]\label{STT}
Let $G$ be supersolvable group, and let $P_i$ be a Sylow $p_i$-subgroup of $G$, 
where $\wp=\{p_1,p_2,\ldots,p_r\}$ constitutes the set of prime divisors of $|G|$. If
the prime divisors are ordered by $p_i>p_{i+1}$ for all $i$, then for each  $k$, 
the product $\prod_{i=1}^kP_i$ 
is a normal subgroup of $G$.
\end{proposition}

 Let $\wp$ be a set of primes, a positive integer $n$ is a \textit{$\wp$-number} 
 if every prime divisor of $n$ belongs to $\wp$. If no prime divisors of $n$ lie in $\wp$, then $n$ 
 will be called a $\wp'$-number.  A positive divisor $d$ of $n$ is a \textit{Hall divisor} 
  if $\gcd(d,n/d)=1$, that is, there exists some set $\wp$ of primes such
 that $d$ is a $\wp$-number while $n/d$ is a $\wp'$-number.
 A group  is a \textit{$\wp$-group} if its order is a $\wp$-number. A subgroup $H$ of a group $G$
  is a \textit{Hall subgroup} if $|H|$ is a Hall divisor of $|G|$. 
  Thus, $H$ is a Hall subgroup of $G$ if and only if,
   for some set $\wp$ of primes, $H$ is a $\wp$-subgroup of $G$ and $|G:H|$ is
  a $\wp'$-number, in which case $H$ is also called a \textit{Hall $\wp$-subgroup} of $G$.
  In particular, if $\wp=\{p\}$ consists of a single prime, then a Hall $\wp$-subgroup is 
  indeed a \textit{Sylow $p$-subgroup}. 

\begin{proposition}[{\rm \cite[Kapitel VI, Satz 1.8]{Huppert1967}}]\label{HALL}
Let $G$ be a solvable group and let $\wp$ be a set consisting of some prime 
divisors of $|G|$. Then the following hold:
\begin{enumerate}[\rm(a)]
\item$G$ contains a Hall $\wp$-subgroup.
\item All Hall $\wp$-subgroups are conjugate in $G$. 
\item Every $\wp$-subgroup of $G$ is contained in some Hall $\wp$-subgroup of $G$.
\end{enumerate}  
\end{proposition}

If $K \le \aut(G)$ and $H \le G$, then $H$ is said to be \textit{$K$-invariant} if
$\sigma(H)=H$ for every $\sigma \in K$. We shall need the following simple lemma.

\begin{lemma}\label{MASCHKE}
Let $N \cong \Z_p^2$ for a prime $p$ and let $K \le \aut(N)$ such that $|K|$ is
coprime to $p$ and $N$ contains a $K$-invariant subgroup of order $p$.
Then $K$ is isomorphic to a subgroup of $\Z_{p-1}^2$.
\end{lemma}
\begin{proof}
Since $|K|$ is coprime to $p$, by Maschke Theorem (see~\cite[12.1.2]{Scott1964}) we have
 $N=N_1 \times N_2$, where $N_1$ and $N_2$ are $K$-invariant subgroups of
  $N$ such that $N_1\cong N_2\cong\Z_p$. Thus, $K\leq\aut(N_1)\times\aut(N_2)\cong\Z_{p-1}^2$.
\end{proof}

Suppose that $G$ acts on a finite set $\Omega$. For an element $x \in \Omega$, we
denote by $G_x$ the \emph{stabilizer} of $x$ in $G$, by $\orb_G(x)$
the \emph{$G$-orbit} containing $x$, and by $\orb_G(\Omega)$ the set of all
$G$-orbits. We denote by $\sym(\Omega)$ the symmetric group consisting
 of all permutation of $\Omega$, and by $\mathrm{id}_\Omega$ the identity permutation.
The following statement is known. It will be used a couple of times in the next 
section, hence we record it here.

\begin{lemma}\label{orbits}
Let $G \le \sym(\Omega)$ be a group containing a regular abelian subgroup $A$,
and suppose that $N \lhd G$. Then
$\orb_N(\Omega)=\orb_B(\Omega)$ for some subgroup $B \le A$.
\end{lemma}
\begin{proof}
It is well known that $\orb_N(\Omega)$ is a block system  for $G$.
Let $K$ be the kernel of the action of $G$ on $\orb_N(\Omega)$. It is 
straightforward to check that $\orb_{K \cap A}(\Omega)=\orb_N(\Omega)$.
\end{proof}

Let $\alpha_i \in \sym(\Omega_i)$, $i=1,2$. The
\emph{direct product} $\alpha_1 \times \alpha_2$ is defined to be
the permutation of $\Omega_1 \times \Omega_2$ acting as
\[
(\alpha_1 \times \alpha_2)((x_1,x_2))=(\alpha_1(x_1),\alpha_2(x_2))~
\text{for all}~(x_1,x_2) \in \Omega_1 \times \Omega_2.
\]

For a prime number $p$, the \emph{affine group} $\agl(1,p)$ consists of
the permutations of the finite field $\F_p$ of the form
$x \mapsto ax+b$, where $a, b \in \F_p$ and $a \ne 0$.
The following result was known already by Galois.

\begin{proposition}[{\rm \cite[Kapitel~II, 3.6~Satz]{Huppert1967}}]\label{GALOIS}
Let $G \le \sym(\Omega)$ be a transitive and solvable group and let $|\Omega|=p$ for a prime $p$.
Then $G$ is isomorphic to a subgroup of $\agl(1,p)$.
\end{proposition}
\subsection{Skew morphisms}

For a skew morphism $\varphi$ of a group $A$ with associated power
function $\pi$, it is well known that the subset 
\[
\Ker\varphi:=\{ x \in A : \pi(x) \equiv 1 \!\!\!\!  \pmod {|\varphi|} \}
\]
is subgroup of $A$, called the \emph{kernel} of $\varphi$~\cite{JS2002}. A skew morphism
$\varphi$ is \emph{kernel-preserving} if $\varphi(\Ker\varphi)=\Ker\varphi$. Moreover,
the subset 
\[
\core\varphi:=\bigcap_{i=1}^{|\varphi|}\varphi^i(\Ker\varphi)
\]
is the core of $A$ in the skew product group $G=A\langle\varphi\rangle$
~\cite[Proposition 6]{HNWY2019}.
It is well known that $\varphi$ is kernel-preserving if and only if $\Ker\varphi=\core\varphi$.
The following properties of $\Ker \varphi$ is well known.

\begin{proposition}[\rm\cite{ CJT2016,CJT2007,JS2002}]\label{CJT1}
Let $\varphi$ be a skew morphism of a finite group $A$ with associated power
function $\pi$. Then we have the following:
\begin{enumerate}[{\rm (a)}]
\item For all $a, b \in A$, $\pi(a)=\pi(a)$ if and only if $a b^{-1} \in \Ker \varphi$.
\item If $A$ is an abelian group, then $\varphi$ is kernel-preserving.
\item If $A$ is non-trivial, then the kernel $\Ker \varphi$ is also non-trivial.
\end{enumerate}
\end{proposition}

The index $|A : \Ker \varphi|$ is called the {\em skew-type} of $\varphi$.
 Note that $\Ker \varphi=A$ if and only if $\varphi$ is
an automorphism of $A$. Using this observation and Proposition~\ref{CJT1}(c), 
we have the following corollary.

\begin{corollary}\label{cor:CJT1}
If $p$ is a prime number, then all skew morphisms of $\Z_p$ are
automorphisms.
\end{corollary}

The direct product of skew morphisms of groups $A$ and $B$, respectively,
may not be a skew morphism of the group $A \times B$.
The following criterion will be useful.

\begin{proposition}[{\rm \cite{Zhang2022}}]\label{Z}
Let $G=A\times B$ and let $\varphi$ and $\psi$ be
skew morphisms of $A$ and $B$ with associated power functions
$\pi_\varphi$ and $\pi_\psi$, respectively.
\begin{enumerate}[{\rm (a)}]
\item The direct product $\varphi \times \psi$ is a skew morphism of $G$ if and only if
\[
\pi_\varphi(a) \equiv \pi_\psi(b) \equiv 1 \!\!\!\! \pmod {d}~\text{for all}~a \in A~
\text{and}~b \in B,
\]
where $d=\gcd(|\varphi|,|\psi|)$.
\item Suppose that $\varphi \times \psi$ is a skew morphism of $G$ with associated
power function $\pi$. Then for all $a\in A$ and $b\in B$,
\[
\pi((a,b)) \equiv \pi_\varphi(a) \!\!\!\!  \pmod {|\varphi|}  \quad\text{and} \quad 
\pi((a,b)) \equiv \pi_\psi(b) \!\!\!\!   \pmod {|\psi|}.
\]
\end{enumerate}
\end{proposition}
Finally, if $\varphi$ is a skew morphism of a group $A$ and $\theta$ is an automorphism of $A$, then
$\theta\varphi\theta^{-1}$ is also a skew morphism of $A$ (see~\cite[Lemma 2.2]{WHYZ2019}), so the automorphism group $\aut(A)$ acts
on the set of skew morphisms of $A$ by conjugation. Two skew morphisms $\varphi$ and $\psi$ of $A$ are \textit{equivalent} if they belong to the same orbit of this action. 
\section{Constructions}\label{sec:cont}
In this section, for certain $n$, we construct non-smooth skew morphisms of the cyclic group $\Z_n$.
The following lemma is very useful.

\begin{lemma}\label{direct}
Let $G=A\times B$ and let $\varphi$ and $\psi$ be skew morphisms of $A$ and $B$, respectively. 
If $\varphi \times \psi$ is a skew morphism of $A \times B$, then $\varphi \times \psi$ is smooth if and only if
both $\varphi$ and $\psi$ are smooth.
\end{lemma}
\begin{proof}
Let $\pi, \pi_\varphi$ and $\pi_\psi$ be the power functions associated with $\varphi \times \psi, \varphi$ 
and $\psi$, respectively.   Assume that $\varphi \times \psi$ is smooth. For any $a \in A$, 
by Proposition~\ref{Z}(b) we have
\[
\pi((\varphi \times \psi)(a, 1_B)) = \pi(a, 1_B) \equiv \pi_{\varphi}(a) \!\!\!\! 
\pmod  {|\varphi|}
\]
 and 
 \[
 \pi((\varphi \times \psi)(a, 1_B)) = \pi(\varphi(a), 1_B) \equiv 
\pi_{\varphi}(\varphi(a)) \!\!\!\! \pmod  {|\varphi|}.
 \]
This implies that 
$\pi_{\varphi}(\varphi(a)) \equiv \pi_{\varphi}(a) \!\! \pmod  {|\varphi|}$, and 
hence $\varphi$ is smooth. Similarly, $\psi$ is also smooth.

Conversely, suppose that both $\varphi$ and $\psi$ are smooth. Then for any $(a, b) \in A \times B$, 
also by Proposition~\ref{Z}(b), we have
\[
\pi((\varphi \times \psi)(a, b)) = \pi(\varphi(a), \psi(b)) \equiv \pi_{\varphi}(\varphi(a))
\equiv  \pi_{\varphi}(a) \equiv  \pi(a,b) \!\!\!\! \pmod  {|\varphi|}
\]
and
\[
\pi((\varphi \times \psi)(a, b)) = \pi(\varphi(a), \psi(b)) \equiv  \pi_{\psi}(\psi(b)) \equiv  \pi_{\psi}(b) \equiv  \pi(a,b) \!\!\!\! \pmod  {|\psi|}.
\]
This implies that  $\pi((\varphi \times \psi)(a, b)) \equiv \pi(a,b) \!\! \pmod  {|\varphi \times \psi|}$, and hence $\varphi \times \psi$  is smooth.
\end{proof}

For positive integers $s$ and $t$, we define $\tau(s,t):=\sum_{i=1}^ts^{i-1}.$ For integers $n$ and 
$r$ with $n \geq 1$, we write $r \in \Z_n^*$ if $\gcd(r,n)=1$.  In this case the \emph{multiplicative order} 
of $r$ in $\Z_n$ is defined to  be the smallest positive integer $l$ for which 
$r^l \equiv 1 \!\! \pmod n$.

\begin{proposition}[{\rm \cite{BJ2017,HNWY2019}}]\label{CSM}
For $n>1$, the proper smooth skew morphisms of $\Z_n$ and the associated power functions are given by the formula
\begin{align}\label{ZSM}
\varphi(x) \equiv x+rk\frac{\tau(s,t)^x-1}{\tau(s,t)-1} \!\!\!\! \pmod{n}~
\text{and}~ \pi(x )\equiv t^x \!\!\!\! \pmod{m},
\end{align}
where $k>1$ is a proper divisor of $n$, $r\in\Z_{n/k}$,  $s\in\Z_{n/k}^*$  and $t \in
\Z_m^*$ are positive integers satisfying the following conditions:
\begin{enumerate}[\rm(a)]
\item $m$ is the smallest positive integer such that 
$r\sum_{i=1}^ms^{i-1} \equiv 0 \!\! \pmod{n/k}$.
\item $t$ has multiplicative order $k$ in $\Z_m$.
\item $s-1 \equiv r\big((\sum_{i=1}^ts^{i-1})^k-1\big)/(\sum_{i=1}^ts^{i-1}-1) 
\!\! \pmod {n/k}$.
\item $s^{t-1} \equiv 1 \!\! \pmod{n/k}$.
\end{enumerate}
Moreover, $k$ is equal to the skew-type of $\varphi$, and $m$ is equal to the order of 
$\varphi$.
\end{proposition}

We remark that, even though all smooth skew morphisms of $\Z_n$ are determined, 
 it is not clear from Proposition~\ref{CSM} that,  for which $n$, the group $\Z_n$ 
underlies a non-smooth skew morphism. We also observe in~\eqref{ZSM} that  the proper smooth 
skew morphisms of $\Z_n$ are all defined by exponential functions on $\Z_n$. 
By contrast, it is well known that every automorphism of $\Z_n$ is a linear function 
of the form $x \mapsto rx$, where $r \in \Z_n^*$. Surprisingly enough, it was shown in~\cite{HKZ2021} 
that the proper skew morphisms of $\Z_n$ which are square roots of automorphisms 
are quadratic polynomials over the ring $\Z_n$. 

\begin{proposition}[\cite{HKZ2021}]\label{ROOT}
Every proper skew morphism $\varphi$ of $\Z_n$
such that $\varphi^2$ is an automorphism of $\Z_n$
 is equivalent to a skew morphism of the form
\[
\varphi(x)\equiv sx-\frac{x(x-1)n}{2k} \!\!\!\! \pmod{n},
\]
where $k$ and $s$ are positive integers satisfying the following conditions:
\begin{enumerate}[\rm(a)]
\item $k^2$ divides $n$ and $s \in \Z_n^*$ if $k$ is odd, and $2k^2$ divides $n$ and
$s\in\Z_{n/2}^*$ if $k$ is even.
\item $s \equiv -1 \!\! \pmod{k}$, $s$ has multiplicative order $2\ell$ in 
$\Z_{n/k}$, and $\gcd(w,k)=1$, where
\[
w \equiv \frac{k}{n}(s^{2\ell}-1)-\frac{s(s-1)}{2} \ell \!\!\!\! \pmod{k}.
\]
\end{enumerate}
The skew-type of $\varphi$ is equal to $k$, and the order of $\varphi$ is equal to $m:=2k\ell$, and finally, the
power function of $\varphi$ is given by 
\[
\pi(x) \equiv 1+2xw'\ell \!\!\!\! \pmod{m},
\]
where $w'$ is determined by the congruence $w'w \equiv 1 \!\! \pmod{k}$.
\end{proposition}

We are ready to show that $\Z_n$ underlies
a non-smooth skew morphism for certain $n$.

\begin{lemma}\label{PNS}
\begin{enumerate}[{\rm (a)}]
\item For any $e\geq 2$, the cyclic group $\Z_{p^e}$ underlies a non-smooth skew morphism,
where $p$ is an odd prime.
\item For any $e\geq 5$, the cyclic group $\Z_{2^e}$ underlies a non-smooth skew morphism.
\end{enumerate}
\end{lemma}
\begin{proof} (a)
In Proposition~\ref{ROOT}, take $n:=p^e$, where $e\geq 2$, and set $k:=p$,
$s:=-1$. Then $\ell=1$,  $w=w'=-1$ and $m=2p$. It is easy to verify that the stated conditions
are satisfied, so we obtain a skew-morphism of $\Z_{p^e}$ and  the associated power function
  given by
 \[
\varphi(x)\equiv-x-p^{e-1}\frac{x(x-1)}{2} \!\!\!\! \pmod{p^e}\quad\text{and}\quad \pi(x)\equiv 1-2x \!\!\!\! \pmod{2p}.
 \]
Since $\pi(\varphi(1))\equiv\pi(-1)\equiv3 \!\! \pmod{2p}$ and 
$\pi(1)\equiv-1 \!\! \pmod{2p}$,
we have $\pi(\varphi(1))\not\equiv \pi(1) \!\! \pmod{2p}$, so $\varphi$ is not smooth.

(b) In Proposition~\ref{ROOT}, let $n:=2^e$ where $e\geq5$, and take $k:=4$ and $s:=-1$.
 As before we have $\ell=1$, $w=w'=-1$ and $m=8$. It is easy to verify that the stated conditions are satisfied, 
 so we obtain a skew morphism of $\Z_{2^e}$ and the associated power function given by
\[
\varphi(x)=-x-2^{e-3}x(x-1) \!\!\!\! \pmod{2^e}\quad\text{and}\quad  \pi(x)\equiv1-2x \!\!\!\! \pmod{8}.
\]
Since $\pi(\varphi(1))\equiv\pi(-1)\equiv3 \!\! \pmod{8}$ and 
$\pi(1)\equiv-1 \!\! \pmod{8}$, 
 $\pi(1)\not\equiv \pi(\varphi(1)) \!\! \pmod{8}$, so $\varphi$ is also not smooth.
\end{proof}

\begin{lemma}\label{NESS}
Let $n$ be a positive integer with a decomposition $n=2^en_1$, where $n_1$ is odd. If $e\geq 5$
or $n_1$ is not square-free, then $\Z_n$ underlies a non-smooth skew morphism.
\end{lemma}
\begin{proof}
Note that $\Z_n \cong \Z_{2^e} \times \Z_{n_1}$.
If $e\geq 5$, then by Lemma~\ref{PNS}(b), the cyclic group $\Z_{2^e}$ has a non-smooth
skew morphism $\alpha$, so by Lemma~\ref{direct}, 
$\varphi:=\alpha\times\mathrm{id}_{n_1}$ is a non-smooth skew morphism
of $\Z_n$, where $\id_{n_1}$ denotes the identity permutation on $\Z_{n_1}$. 
If $n_1$ is not square-free, then there is an odd prime $p$ such that $p^e\mid n_1$ where $e\geq 2$,
using Lemma~\ref{PNS}(a) and similar techniques we can construct a non-smooth skew morphism
of $\Z_n$, as required.
\end{proof}

\section{Proof of Theorem~\ref{main1}}\label{sec:main1}
The aim of this section is to prove Theorem~\ref{main1}. Before doing this, we make the 
following convention. Suppose that $G=AY$ is a complementary factorization, 
 where the subgroup $Y$ is cyclic and core-free in $G$. Let $y$ be a fixed  generator 
 of $Y$, then a skew morphism $\varphi$ of $A$ together with the associated power function $\pi$ are defined according to \eqref{eq:c-rule}. The skew morphism $\varphi$ will be referred to as the \emph{skew morphism of $A$ induced by $y$}.  Note that the action of $G$ on the set $\Omega:=\{gY\mid g\in G\}$ of left cosets of $Y$ in $G$ is faithful; in particular, the subgroup $A$ acts regularly on $\Omega$ and $Y$ is a point stabilizer. Therefore, we may  assume  that $G \le \sym(\Omega)$ with $G=AY$, where $A$ is a regular subgroup of $G$ and $Y=G_x$ is cyclic for some $x \in \Omega$. This assumption will be used throughout this section and for the sake of convenience we record it here.

\begin{assumption}\label{conv}
Let $G \le \sym(\Omega)$ be a group such that
$G=AY$, where $A$ is abelian and regular on $\Omega$, $Y=G_x=\sg{y}$ 
for a fixed element $x\in \Omega$, and let $\varphi$ be the skew morphism
induced by $y$.
\end{assumption}

The proof of Theorem~\ref{main1}  will be given after four preparatory lemmas.
\begin{lemma}\label{L1}
With the notations in Assumption~\ref{conv}, 
suppose  that $O_p(G) \le A$ for every prime divisor $p$ of $|G|$. Then $A \lhd G$.
\end{lemma}
\begin{proof}
Denote by $\wp=\{p_1,p_2,\ldots,p_k\}$ the set of prime divisors of $|G|$,
then by Proposition~\ref{FITTING}(a) $F(G)=\prod_{i=1}^{k}O_{p_i}(G) \le A$.
On the other hand, by Proposition~\ref{AB}(a) $G$ is solvable, and by Proposition~\ref{FITTING}(b) 
we have $A \le C_G(F(G)) \le F(G)$. Therefore $A=F(G)$, in particular, $A \lhd G$.
\end{proof}

For a prime divisor $p$ of $|A|$, we denote by $A_p$ the (unique) Sylow $p$-subgroup of the abelian group $A$ and by $A_{p'}$ the (unique) Hall $p'$-subgroup of $A$.

\begin{lemma}\label{L2}
 With the notations in Assumption~\ref{conv}, suppose that $O_p(G) \ne 1$ for 
 some prime divisor $p$ of $|A|$ and   $|A_p|=p$.
Then $O_p(G)=A_p$ or $A_p \times Y^*$, where $Y^* \le Y$ and $|Y^*|=p$.
\end{lemma}
\begin{proof}
Write $N:=O_p(G)$. By Lemma~\ref{orbits}  we know that $\orb_N(\Omega)=\orb_{B}(\Omega)$
for some $B\leq A$. Since the size of each $N$-orbit is a power of $P$, we
have $B=A_p$. Thus, $\orb_N(\Omega)$ is a block system for $G$ in which each 
block has size $p$. Let $K$ be the kernel of the action of $G$ on $\orb_N(\Omega)$. 
By Proposition~\ref{AB}(a) $G$ is solvable, so $K$ is solvable as well. It follows from 
Proposition~\ref{GALOIS}  that $K$ is 
 isomorphic to a subgroup of the direct product  
$\agl(1,p) \times \cdots \times \agl(1,p)$ with $|A|/p$ factors. It is clear that the 
Sylow $p$-subgroup $P$ of $K$ is normal in $K$ and it is elementary abelian. 
Thus, $P\ch K\lhd G$ and hence $P\lhd G$, we obtain $P\leq O_p(G)$. 
On the other hand, since $O_p(G)=N \leq K$, we also have $O_p(G)\leq P$. Therefore, $O_p(G)=P$, and 
we conclude that $O_p(G)=A_p$ or $A_p \times Y^*$, as required.
\end{proof}

\begin{lemma}\label{L3}
 With the notations in Assumption~\ref{conv}, suppose that $N$ is a normal subgroup of $G$. For $g \in G$, 
 denote by $\bar{g}$ 
 the image of $g$ under its action on $\orb_N(\Omega)$ and for $H \le G$, set $\bar H=\{\bar h : h \in H\}$.
\begin{enumerate}[{\rm (a)}]
\item $\bar G=\bar A\bar Y$, $\bar A\cap\bar Y=1$ and $\bar Y$ is core-free in
$\bar G$.
\item Let $\bar\varphi$ be the skew morphism of $\bar A$ induced by $\bar y$.
Then
\[
\bar\varphi( \bar a)=\overline{(\varphi(a))}~\text{for all}~a \in A.
\]
\item Let $\bar\pi$ be a power function associated with $\bar\varphi$. Then
\[
\bar\pi(\bar a) \equiv \pi(a) \!\!\!\!  \pmod {|\bar\varphi|}~\text{for all}~a\in A.
\]
\end{enumerate}
\end{lemma}
\begin{proof}
The mapping $g \mapsto \bar g$ is an epimorphism from $G$ onto $\bar{G}$,
therefore, $\bar G=\bar A \bar Y$. Since $\bar A$ is abelian
and transitive on $\orb_N(\Omega)$, it is regular. It is evident that 
$\bar Y$ is contained in the stabilizer  
of the $N$-orbit containing the identity element $1_A$. 
Denoting this stabilizer by $Z$, we have $\bar A\bar Y=\bar G=\bar A Z$. 
Since $\bar A$ is regular, we have $\bar A\cap\bar Y=\bar 1=\bar A\cap Z$, so 
$Z=\bar Y$. This proves (a).

Fix any $a \in A$, we have
$ya=\varphi(a) y^{\pi(a)}$, so $\bar y \bar a=
\overline{\varphi(a)} (\bar y)^{\pi(a)}$. Then \eqref{eq:c-rule} can be applied to the factorisation
$\bar A \bar Y$ and generator $\bar y \in \bar Y$, giving rise to the equality
$\bar\varphi(\bar a)=\overline{\varphi(a)}$ and the congruence
$\bar\pi(\bar a) \equiv \pi(a) \pmod {|\bar\varphi|}$, this proves (b) and (c).
\end{proof}

In what follows,  the skew morphism $\bar\varphi$ defined as in the lemma above will be called \emph{the quotient of  $\varphi$ determined by $N$}, and denoted by 
$\varphi_N$.

\begin{lemma}\label{L4}
With the notations in Assumption~\ref{conv}, 
suppose that $G$ has two distinct normal subgroups $M$ and $N$ satisfying the following
conditions:
\begin{enumerate}[{\rm (a)}]
\item There are subgroups $B,\, C \le A$ such that $B \cap C=1$, and
\[
\orb_M(\Omega)=\orb_B(\Omega)~\text{and}~\orb_N(\Omega)=\orb_C(\Omega).
\]
\item Both skew morphisms $\varphi_M$ and $\varphi_N$ are smooth.
\end{enumerate}
Then $\varphi$ is smooth.
\end{lemma}
\begin{proof}
For any $a \in A$, since  $\varphi_M$ is smooth, by Lemma~\ref{L3}(a)--(b), we have
\[
\pi(\varphi(a)) \equiv  \bar\pi(\overline{\varphi(a)})=\bar\pi(\varphi_M(\bar a))\equiv  \bar\pi (\bar a)\equiv\pi(a) \!\!\!\! \pmod { |\varphi_M|}.
\]
Similarly, 
$\pi(\varphi(a)) \equiv \pi(a)  \!\! \pmod  {|\varphi_N|}.$

In what follows we show that $|\varphi|=\lcm(|\varphi_M|,|\varphi_N|)$, which together with the above congruences 
will imply $\pi(\varphi(a))\equiv\pi(a)\!\! \pmod{|\varphi|}$ for all $a\in A$, and hence $\varphi$ is smooth. 
Indeed, let $K$ and $L$ be the kernels of the actions of $G$ on
$\orb_M(\Omega)$ and $\orb_N(\Omega)$, respectively. Then
\begin{equation}\label{eq:con3}
|Y|=|\varphi_M||Y \cap K|= |\varphi_N||Y \cap L|.
\end{equation}
 Let $d=\gcd(|\varphi_M|,|\varphi_N|)$. Now  by \eqref{eq:con3} we have
\begin{equation}\label{eq:con33}
 \frac{|\varphi_M|}{d}\cdot |Y\cap K|= \frac{|\varphi_N|}{d}\cdot |Y\cap L|.
\end{equation}
 If $\psi \in (Y \cap K)\cap (Y \cap L)$, then for any $x \in \Omega$,
 $\psi$ fixes the orbits $\orb_M(x)$ and $\orb_N(x)$, so by (a) we have
 $\psi(x) \in \orb_B(x) \cap \orb_C(x)$. Since $B\cap C=1$ and $A$ is regular,
 we have $\orb_B(x) \cap \orb_C(x)=\{ x \}$, and so $\psi=\id_\Omega$. But $Y$ is a cyclic group, 
 we obtain $\gcd(|Y \cap K|,|Y \cap L|)=1$, and from \eqref{eq:con33} we deduce that
 $|Y\cap K|$ divides $|\varphi_N|/d$.  On the 
other hand, $|\varphi_N|/d$ divides $|Y\cap K|$ because it is coprime to 
$|\varphi_M|/d$, and we find $|Y\cap K|=|\varphi_N|/d$. Thus, 
by  \eqref{eq:con3} $|\varphi|=|Y|=\lcm(|\varphi_M|,|\varphi_N|).$
 \end{proof}

\begin{proof}[Proof of Theorem~\ref{main1}]
We keep the notations set in Assumption~\ref{conv} and, in addition, assume 
that $A \cong \Z_n$ for some $n \ge 1$. We have to show that $\varphi$ is smooth if 
and only if $n=2^en_1$, where $0 \le e \le 4$ and $n_1$ is an odd 
square-free number. 
The necessity has been proved in Lemma~\ref{NESS}. For the converse, we proceed by induction on $n$. 
If $n=2^e$ $(0\leq e\leq 4)$,  then by the
census in~\cite{BACH2022} we know that $\varphi$ is smooth; if $n$ is an odd prime, 
then $\varphi$ is an automorphism of $A$ due to Corollary~\ref{cor:CJT1},
in particular, it is smooth.

From now on we will assume that $n$ is either an odd composite number or an
even number with an odd prime divisor and also that the theorem holds for any
cyclic group whose order is a proper divisor of $n$.
By Proposition~\ref{AB}(b), $G$ is supersolvable, and therefore,
Proposition~\ref{STT} can be applied to $G$. In particular, if $p$ is the largest prime
divisor of $|G|$, then  the Sylow $p$-subgroup $P$ of $G$ is normal in $G$.
If $O_q(G) \ne 1$ for another prime divisor $q$ of $|G|$, then
Lemma~\ref{L4} can be applied to $G$ with $M=P$ and
$N=O_q(G)$. Using also the induction hypothesis, we obtain that
$\varphi$ is smooth.

Thus we may assume that $F(G)=O_p(G)=P$. Notice that $p > 2$, so 
$|A_p|=p$. By Lemma~\ref{L2} we have $P=A_p$ or $P=A_p \times Y^* \cong \Z_p^2$, where $Y^*\leq Y$ and $|Y^*|=p$.
If $P=A_p$, then from Lemma~\ref{L1} we deduce that $A \lhd G$, and so $\varphi$ is an 
automorphism; in particular, $\varphi$ is smooth.  In what follows we assume $P=A_p\times Y^*$. 

By Proposition~\ref{HALL} we may assume that $Q$ is a Hall $p'$-subgroup of $G$ containing $A_{p'}$. 
Then $G=P\rtimes Q$ and by Proposition~\ref{FITTING} $C_G(P)=C_G(F(G))=F(G)=P$. It follows that
 $Q$ is isomorphic to a subgroup of $\aut(P)$. By Proposition~\ref{CJT1}(b)--(c), 
 $1\neq\Ker \varphi \lhd G$, and so $1\neq\Ker\varphi\leq F(G)\cap A=A_p$. But $|A_p|=p$, we get 
 $A_p=\Ker\varphi$ and $A_p$ is a normal subgroup of $G$. Thus, $A_p$ is $Q$-invariant. 
 Since  $p\nmid |Q|$, by Lemma~\ref{MASCHKE} the subgroup $Q$ is isomorphic to 
 a subgroup of $\Z_{p-1}^2$, so it is abelian. Now  $Q \cong G/P \cong \bar{G}$, where $\bar{G}$ is the image of $G$ induced by its action on $\orb_P(\Omega)$. By Lemma~\ref{L3} $Q\cong \bar A\bar Y$ and the latter 
group is the skew product group induced by $\bar A$ and the quotient skew morphism $\bar\varphi$. 
 Since it is abelian, $\bar Y=1$, so $Y\leq P$. Therefore, $|\varphi|=|Y|=p$, and hence $\varphi$ is smooth.
\end{proof}
\section{non-cyclic abelian groups}\label{sec:main2}
In this section, on the basis of the results obtained in the previous sections, we shall give a partial solution to 
Problem~\ref{PROB}. The following classification of proper skew morphisms of the elementary abelian group
 $\Z_p\times\Z_p$ will be useful.
\begin{proposition}[{\rm \cite[Theorem 5.10]{CJT2016}}]\label{NSE}
Let $(a,x)$ be a basis of  the elementary abelian group $A\cong\Z_p\times \Z_p$, where $p$ is an odd prime. 
Then every  proper skew morphisms of $A$ can be expressed as the  form
 \[
 \varphi(a^ix^j)=a^{ri+\frac{1}{2}dj(j-1)rn}(bx^r)^j,
 \]
where $b$ is an element in the kernel $\Ker\varphi=\langle a\rangle$ uniquely 
determined by the triple $(d,n,r)$, for all $d,n \in \Z_p^*$ and $r\in\{2,\ldots, p-1\}$. 
The skew morphism $\varphi$ has  order $pk$, where $k$ is the multiplicative order 
of $r$ modulo $p$, and its power function is given by
 \[
 \quad\pi(a^ix^j)\equiv 1+jnk \!\!\!\! \pmod{pk}.
 \]
 \end{proposition}
\begin{lemma}\label{NON2}Let $p$ be an odd prime, then every proper skew 
morphism of the elementary abelian group $\Z_p\times\Z_p$ is non-smooth.
\end{lemma}
\begin{proof}
Using the notation in Proposition~\ref{NSE}, we have $\pi(x)\equiv 1+nk \!\! \pmod{pk}$ 
and $\pi(\varphi(x))\equiv\pi(bx^r)\equiv\pi(x^r)\equiv1+rnk \!\! \pmod{pk}$,
and so $\pi(x)\not\equiv\pi(\varphi(x))\!\! \pmod{pk}$. Therefore, $\varphi$ is not smooth.
\end{proof}

\begin{lemma}\label{ZpZpNS}
Let $G = A \times B$ be an abelian group such that $A\cong\Z_p\times \Z_p$, 
where $p$ is an odd prime. Then  $G$ underlies a non-smooth skew morphism.
\end{lemma}
\begin{proof}
By Lemma~\ref{NSE}, the elementary abelian group $A\cong\Z_p\times \Z_p$ has a non-smooth
skew morphism $\alpha$, so by Lemma~\ref{direct} $\varphi:=\alpha\times\mathrm{id}_{B}$ 
is a non-smooth skew morphism of $G$.
\end{proof}

\begin{proof}[Proof of Theorem~\ref{main2}]
If the Sylow $2$-group $P$ of $A$ contains  a direct factor isomorphic to $\mathbb{Z}_{2^e}$
 for some $e\geq 5$, then by Lemma~\ref{direct} and  Lemma~\ref{PNS} we can construct 
 a non-smooth skew morphism of $A$. 

If  $n_1$ is not square-free.  Then there is an odd prime $p$ such that $p^2\mid n$, and so
  either $A\cong\mathbb{Z}_p\times\Z_p\times B$,  or $A\cong\Z_{p^i}\times B$ for some $i\geq 2$
   and some subgroup $B\leq A$. By Lemmas~\ref{PNS}(a) and~\ref{ZpZpNS},
  both $\Z_p\times\Z_p$ and $\Z_{p^i}$ underly  a non-smooth skew morphism $\alpha$. Therefore,
by Lemma~\ref{direct}  the skew morphism $\varphi=\alpha\times \id_{B}$  is a non-smooth skew morphism of $A$.
\end{proof}

Let $A$ be a non-cyclic abelian group of order $n=2^fn_1$, where $n_1$ is odd. 
If every skew morphism of $A$ is smooth. Then, by Theorem~\ref{main2},  $n_1$ is square-free and 
the Sylow $2$-subgroup $P$ of $A$ contains no direct factors isomorphic to $\mathbb{Z}_{2^e}$ for any $e\geq 5$.
Thus,  $P\cong\Z_{2^{e_1}}\times\cdots\times \Z_{2^{e_r}}$ for some positive integer $r$, where $1\leq e_i\leq 4$ for all $1\leq i\leq r$. 
In particular, if $e_1=e_2=\cdots=e_r=1$, it is known that every skew morphism of $P$
is an automorphism (see~\cite[Theorem~5.8]{CJT2016}), so it is smooth. But for other cases, little is known.
For an abelian $2$-group $\Z_{2^{e_1}}\times\cdots\times \Z_{2^{e_r}}$, we call the $r$-tuple $(e_1,\ldots, e_r)$  \textit{admissible} if  every
skew morphism of the abelian group  $\Z_{2^{e_1}}\times\cdots\times \Z_{2^{e_r}}$ is smooth.
 \begin{problem}
Determine the admissible $r$-tuples for all $r\geq 1$.
 \end{problem}

  \begin{problem}
 Suppose that $(e_1,\ldots, e_r)$ is an admissible $r$-tuple, and $A$ is an abelian group such that its Sylow $2$-subgroup
 $P\cong\Z_{2^{e_1}}\times\cdots\times \Z_{e^r}$ and the order of its Hall $2'$-subgroup is square-free, is it true that every
 skew morphism of $A$ is smooth?
   \end{problem}


\end{document}